\documentclass[a4paper,]{article}
\usepackage{lineno,hyperref,vmargin}
\setpapersize{A4}
\usepackage{amssymb,amsmath,amsthm, array, appendix}
\usepackage{dsfont}
\usepackage[normalem]{ulem}
\usepackage{url}

\newcommand{\ieg}{\left[\hspace{-0.9ex}\left[\hspace{0.5ex}}
\newcommand{\ied}{\hspace{0.5ex} \right]\hspace{-0.9ex}\right]}
\newcommand{\argmax}{\operatornamewithlimits{argmax}}

\newcommand{\relmiddle}[1]{\mathrel{}\middle#1\mathrel{}}
\newtheorem{theorem}{Theorem}
\newtheorem{lemma}{Lemma}

\newtheorem{definition}{Definition}
\newtheorem{proposition}{Proposition}
\newtheorem{corollary}{Corollary}

\newtheorem{hypothesis}{Hypothesis}

\newcommand{\g}[1]{\boldsymbol{#1}}

\newcommand{\F}[0]{\mathcal{F}}

\modulolinenumbers[5]

\begin{document}

\title{Rademacher Complexity and Generalization Performance of Multi-category Margin Classifiers}

\author{Khadija Musayeva, Fabien Lauer and Yann Guermeur}

\maketitle

\begin{abstract}
One of the main open problems in the theory of multi-category margin classification is the form of the optimal dependency of a guaranteed risk on the number $C$ of categories, 
the sample size $m$ and the margin parameter $\gamma$. From a practical point of view, the theoretical analysis of generalization performance contributes to the development of new learning algorithms.
In this paper, we focus only on the theoretical aspect of the question posed. More precisely, under minimal learnability assumptions, we derive a new risk bound for multi-category margin classifiers. 
We improve the dependency on $C$ over the state of the art when the margin loss function considered satisfies the Lipschitz condition. 
We start with the basic supremum inequality that involves a Rademacher complexity as a capacity measure. 
This capacity measure is then linked to the metric entropy through the chaining method. In this context, our improvement is based on the introduction of a new combinatorial metric entropy bound.
\end{abstract}

\section{Introduction}\label{sec:introduction} 
Although the theory of binary pattern classification is well established \cite{Vap98, Bar98}, the theory of multi-category classification is far from being complete. The research in this case addresses problems such as the sample-complexity analysis of empirical risk minimization algorithms \cite{DanSabBenSha11}, or consistency analysis of multi-class loss functions and of specific families of classifiers \cite{DogGlaIge16}. Another open question is the optimal dependency of guaranteed risks of multi-category classifiers on the number 
$C$ of categories and the sample size $m$. It is all the more the case for the problems that involve a large number of classes. When the considered classifiers are margin ones that take decision based on a score per category, 
the dependency on the margin parameter $\gamma$ also becomes relevant to the characterization of their generalization performance.
If this question has been mainly studied for specific families of classifiers, be it $k$-nearest neighbors \cite{KonWei14}, kernel methods \cite{Zha04, LeiDogBinKlo15} and decision trees \cite{KuzMohSye14}, 
tackling it under minimal learnability assumptions remains a challenging task. This paper focuses on obtaining guaranteed risks under such assumptions.

The first step in the derivation of risk bounds is the choice of the margin loss function.
Two families of margin loss functions can be distinguished: indicator margin loss functions and those that satisfy the Lipschitz condition. Deriving guaranteed risks with the optimal dependency on the parameters of interest is relatively straightforward in the first case \cite{Gue17}. The family of Lipschitz continuous loss functions, on the other hand, offers a richer setting to this task. In this case, one can obtain a guaranteed risk whose control term involves a Rademacher complexity \cite{KolPan02}. Then a sequence of transitions between capacity measures is performed. More precisely, using the chaining method one can control the Rademacher complexity of a function class through the sum of its metric entropies \cite{Tal14}.
A~combinatorial bound is then used to estimate the metric entropy of the class in terms of its combinatorial dimension. 
In this sequence of transitions, one can choose the capacity measure at the level of which to reduce the multi-class problem to an ensemble of bi-class ones, that is, to perform a \textit{decomposition}. 
Performing a decomposition for Rademacher complexity, a linear dependency on $C$ was obtained in \cite{KuzMohSye14}. This dependency has been improved to a sublinear one in \cite{Gue17} by postponing the decomposition to the level of metric entropy.

In this paper, we exactly follow the pathway of \cite{Gue17}. Our contribution is based on the following line of reasoning. 
Theorem~7 of \cite{Gue17} provides a sublinear (but still close to linear) dependency on $C$ using a decomposition result for metric entropies (Lemma~1 of \cite{Gue17}) in $L_p$-norm with $p=2$ 
and the combinatorial metric entropy bound of \cite{MenVer03}. On the other hand, using the decomposition result with $p=\infty$ and the $L_\infty$-norm metric entropy bound of \cite{AloBenCesHau97}, one can obtain a radical
dependency on $C$, this, however, at the expense of a degraded dependency on $m$. Hence, we consider the values of $p$ in between these two extreme ones, and extend the $L_2$-norm bound of \cite{MenVer03} to $L_p$-norms
with integer $p >2$. When applied in the chaining, it results in
an improved dependency on $C$ over that of Theorem~7 of \cite{Gue17}. Specifically, we obtain a radical dependency on $C$ (up to logarithmic factors) without worsening the dependencies on $m$ and $\gamma$.

The organization of the paper is as follows. In the next section, we introduce the theoretical framework and describe the transitions between the capacity measures. Then, Section~\ref{sec:metric-entropy-bound} 
gives the new combinatorial metric entropy bound, whose proof can be found in~\ref{app:proofThm2}. In Section~\ref{sec:bound-Rademacher-complexity}, we demonstrate how 
this result can be applied in the chaining to derive an improved upper bound on the Rademacher complexity. 
Conclusions and ongoing research are highlighted in Section~\ref{sec:conclusions}.
All intermediate results used in the proofs are collected in~\ref{app:intermediate}.

\paragraph{Notation} We denote the set of strictly positive reals by $\mathbb{R}_+$, and let $\mathbb{N}^{*}=\mathbb{N} \setminus \{0\}$. $\ieg i, j \ied$ stands for the set of integers from $i$ to $j$. $\mathds{1}_{A}$ stands for the indicator function for the event $A$ such that $\mathds{1}_{A}=1$ if $A$ occurs, and $0$ otherwise.
$\lfloor x \rfloor$ is the greatest integer less than or equal to $x$, $\lceil x \rceil$ is the smallest integer greater than or equal to $x$.

\section{Theoretical Framework}\label{sec:theoretical-framework}
We consider $C$-category pattern classification problems with $C \geqslant 3$. Each object is represented by its description $x \in \mathcal{X}$ and the categories $y$ belong to $\mathcal{Y}=\ieg 1, C \ied$.
We assume that $\left(\mathcal{X}, \mathcal{A}_{\mathcal{X}}\right)$ and $\left(\mathcal{Y}, \mathcal{A}_{\mathcal{Y}}\right)$ are measurable spaces. Denote by $\mathcal{A}_{\mathcal{X}} \otimes \mathcal{A}_{\mathcal{Y}}$ the product sigma-algebra on $\mathcal{X} \times \mathcal{Y}$.
We assume that the link between descriptions and categories can be characterized by an unknown probability
measure $P$ on the measurable space $\left(\mathcal{X} \times \mathcal{Y},  \mathcal{A}_{\mathcal{X}} \otimes \mathcal{A}_{\mathcal{Y}} \right)$.
Let $Z = \left ( X,Y \right )$ be a random pair with values in $\mathcal{Z}=\mathcal{X} \times \mathcal{Y}$,
distributed according to $P$. The available information on $P$ is limited to an $m$-sample
$\mathbf{Z}_m = \left( Z_i  \right)_{1 \leqslant i \leqslant m} =\left( \left ( X_i, Y_ i \right) \right)_{1 \leqslant i \leqslant m}$ distributed according to $P^m$.
In the following, we distinguish the sample size $m$ from the generic notation $n$ which stands for a number of points in a set that needs not be a realization of a random sample. 

We consider multi-category margin classifiers that take their decisions based on a score per category and focus on those that implement classes of functions with values in a hypercube of $\mathbb{R}^C$
(thus, in contrast to \cite{LeiDogBinKlo15}, no correlation assumption is made on the component functions). Most well-known classifiers, such as neural networks \cite{AntBar99}, support vector machines \cite{DogGlaIge16}, and nearest neighbors \cite{KonWei14} are margin classifiers.
\begin{definition}[Multi-category margin classifiers]
\label{def:margin-multi-category-classifiers}
Let $\mathcal{G} = \prod_{k=1}^C \mathcal{G}_k$ be a class of
functions from $\mathcal{X}$ into $\left [-M_{\mathcal{G}}, M_{\mathcal{G}} \right ]^{C}$
with $M_{\mathcal{G}} \in \left [ 1, +\infty \right )$.
For each $g = \left ( g_{k} \right )_{1 \leqslant k \leqslant C}
\in \mathcal{G}$, $dr_g$ is a {\em multi-category margin classifier} such that for all $x \in \mathcal{X}$, $dr_g(x)=\argmax_{1 \leqslant k \leqslant C}g_k(x)$,
breaking ties with a dummy category $*$.
\end{definition}
To sidestep the complications that might arise from the measurability of a supremum of an uncountable set, we assume that the classes $\mathcal{G}_k$, and in general, all sets of functions considered in the sequel satisfy the ``image admissibility Suslin'' condition \cite[page 101]{Dud84}.

The classification performance of margin classifiers can be characterized based on the following functions.
\begin{definition}[Class  $\mathcal{F}_{\mathcal{G}}$ of margin functions]
\label{def:class-of-transformed-functions}
Let $\mathcal{G}$ be as in Definition~\ref{def:margin-multi-category-classifiers}.
For any $g \in \mathcal{G}$, the margin function $f_{g}:\mathcal{Z} \rightarrow [-M_{\mathcal{G}}, M_{\mathcal{G}}]$ is
$$
\forall (x,k) \in \mathcal{Z}, \quad f_{g} \left(x, k\right) = \frac{1}{2} \left( g_k \left ( x \right )\!-\!\max_{l \neq k} g_l \left ( x \right ) \right).
$$
Then, we define $\mathcal{F}_{\mathcal{G}}=\left\{ f_{g} : g \in \mathcal{G} \right\}.$
\end{definition}
Given $g \in \mathcal{G}$, $dr_g$ misclassifies $\left(x,y\right)$ if $dr_g(x) \neq y$, or equivalently,
if $f_g\left(x,y\right) \leqslant 0$. The goal of the learning process is to minimize the probability of error or \textit{risk} over $\mathcal{G}$.
\begin{definition}[Risk $L$]\label{def:risk}
Let $\mathcal{G}$ be as in Definition~\ref{def:margin-multi-category-classifiers}.
Let $\phi$ be the standard indicator loss function defined as
$$
\forall t \in \mathbb{R}, \;  \phi(t)=\mathds{1}_{\{t\leqslant 0\}}.
$$
For any $g \in \mathcal{G}$, its risk $L(g)$ is
$$L(g)=\mathbb{E}_{Z}\left[\phi\left(f_{g}\left(Z\right)\right)\right]=P\left(dr_g(X) \neq Y\right).$$
\end{definition}
To make use of the values of functions $f_g$ (and not just of their signs) in the assessment of the classification performance, we appeal to the following margin loss function.
\begin{definition}[Parameterized truncated hinge loss function $\phi_\gamma$]
\label{def:truncated-hinge-loss}
For any $\gamma \in (0, 1]$, the parameterized truncated hinge loss function $\phi_{\gamma}$ is defined as
$$
\forall t \in \mathbb{R}, \;  \phi_{\gamma}(t)=\mathds{1}_{\{t\leqslant 0\}}+\left(1-\frac{t}{\gamma}\right)\mathds{1}_{\{t \in (0, \gamma]\}}.
$$
\end{definition}
It is clear from the definition that $\phi_{\gamma}$ dominates the standard indicator loss function given in Definition~\ref{def:risk} and that it is Lipschitz continuous. 
Observe that when this loss function is applied to $f_{g}$, the values of the latter strictly above $\gamma$ and below zero become irrelevant 
to the estimation of the classification accuracy. Taking benefit from this fact, we introduce functions $f_{g, \gamma}$ by restricting the codomain of $f_{g}$ to $[0, \gamma]$ for all $g \in \mathcal{G}$. In \cite{KuzMohSye14}, a partial restriction is the main source of improvement upon the result of \cite{KolPan02} in terms of the dependency on $C$. The use of the set of functions $f_{g, \gamma}$ leads to even a finer bound, this time in terms of the diameter of the function class as we switch from $2M_{\mathcal{G}}$ to $\gamma$.
\begin{definition}[Class $\mathcal{F}_{\mathcal{G},\gamma}$ of truncated margin functions]
\label{def:class-of-transformed-functions2}
Let $\mathcal{F}_{\mathcal{G}}$ be a class of functions satisfying
Definition~\ref{def:class-of-transformed-functions}.
Fix $\gamma \in (0,1]$. For any $f_{g} \in \mathcal{F}_{\mathcal{G}}$, we define $f_{g, \gamma}: \mathcal{Z} \rightarrow [0, \gamma]$ as
$$
\forall (x,k) \in \mathcal{Z}, \; f_{g,\gamma}\left(x, k \right) = \max\!\left(0, \min\left(\gamma, f_{g}\left(x, k \right) \right)\! \right),
$$
and $\mathcal{F}_{\mathcal{G},\gamma}=\{f_{g,\gamma} : g \in \mathcal{G} \}$.
\end{definition}
For any $g \in  \mathcal{G}$, its risk, $L(g)$ 
can be upper bounded by the margin risk $L_{\gamma}(g)$ obtained on the basis of the loss function $\phi_\gamma$.
It is the $m$-sample $\mathbf{Z}_m$ based estimate of $L_{\gamma}$ that appears in our guaranteed risk.
\begin{definition}[Margin risk $L_{\gamma}$ and empirical margin risk $L_{\gamma,m}$]
\label{def:empirical-margin-risk}
Let $\mathcal{G}$ be a class of functions satisfying
Definition~\ref{def:margin-multi-category-classifiers}. Let $\phi_\gamma$ be as in Definition~\ref{def:truncated-hinge-loss}.
Then, for $\gamma \in (0, 1]$, the margin risk $L_{\gamma}$ associated with any $g \in \mathcal{G}$ is
$$L_{\gamma}(g)=\mathbb{E}_{Z}\left[\phi_{\gamma}\left(f_{g,\gamma}\left(Z\right)\right)\right].$$
Its $m$-sample $\mathbf{Z}_m$ based estimate is the empirical margin risk defined as
$$L_{\gamma,m}(g)=\frac{1}{m} \sum_{i=1}^m \phi_{\gamma}\left(f_{g,\gamma}\left(Z_i\right)\right).$$
\end{definition}



In what follows, we give the definitions of the capacity measures we use and outline the transitions between them, which are at the basis of the derivation of our result. 
We use $\mathcal{F}$ to denote a uniformly bounded class of functions on a generic measurable space $\left(\mathcal{T}, \mathcal{A}_\mathcal{T}\right)$.
First, we recall the definition of the Rademacher complexity. 
\begin{definition}[Rademacher complexity]
Let $P_\mathcal{T}$ be a probability measure on $\left(\mathcal{T}, \mathcal{A}_\mathcal{T}\right)$ and $\mathbf{T}_n=(T_i)_{1 \leqslant i \leqslant n}$ a sequence of independently distributed according to  $P_\mathcal{T}$ random variables 
with values in $\mathcal{T}$. Let $\boldsymbol{\sigma}_n=\left(\sigma_i \right )_{1 \leqslant i \leqslant n}$
be a Rademacher sequence, i.e.,a sequence of independent random variables uniformly distributed in $\{-1,+1\}$.
Then, the empirical Rademacher complexity of $\mathcal{F}$ given $\mathbf{T}_n$ is defined as
$$
\hat{R}_n \left (\mathcal{F}\right ) = 
\mathbb{E}_{\boldsymbol{\sigma}_n}
\left [ \sup_{f \in\mathcal{F}} \frac{1}{n}
\sum_{i=1}^n \sigma_i f \left ( T_i \right )
\relmiddle| \mathbf{T}_n \right ]
$$
and its  Rademacher complexity is
$ R_n \left (\mathcal{F} \right ) 
= \mathbb{E}_{\mathbf{T}_n}  \left [ \hat{R}_n \left (\mathcal{F}\right) \right ].$
\end{definition}

The capacity measures central in the derivation of our result are covering/packing numbers. Their
definitions require the introduction of the following empirical pseudo-metrics: for any $f,f' \in\mathcal{F}$ and $\mathbf{t}_n =(t_i)_{1\leqslant i\leqslant n}\in \mathcal{T}^n$, 
$$
\displaystyle{d_{p, \mathbf{t}_n} ( f, f')=
\begin{cases}
\left ( \frac{1}{n}\sum_{i=1}^n
\left| f (t_i)-f'(t_i)\right|^p\right)^{\frac{1}{p}}, \mbox{if } p \in [1, +\infty) \\
\max_{1 \leqslant i \leqslant n} \left |f (t_i)-f'(t_i)\right |,
\mbox{if } p =+\infty.
\end{cases}}
$$ 
\begin{definition}[Covering numbers, metric entropy, packing numbers]
\label{def:coverin-packing-number}
The $L_p$-norm {\em $\epsilon$-covering number} of $\mathcal{F}$, $\mathcal{N} \left ( \epsilon, \mathcal{F}, d_{p,\mathbf{t}_n} \right)$, 
is the smallest cardinality of the $\epsilon$-nets of $\mathcal{F}$, i.e., subsets $\bar{\mathcal{F}}\subseteq\mathcal{F}$ such that $\forall f\in\mathcal{F}$ there exists $\bar{f} \in \bar{\mathcal{F}}$ such that $d_{p,\mathbf{t}_n}(f, \bar{f}) < \epsilon$. 
The logarithm of $\mathcal{N} \left ( \epsilon, \mathcal{F}, d_{p,\mathbf{t}_n} \right)$ is the {\em metric entropy} of $\mathcal{F}$.
A subset $\bar{\mathcal{F}}$ of $\mathcal{F}$ is {\em $\epsilon$-separated} with respect to $d_{p, \mathbf{t}_n}$ if, for any two distinct elements $f, f^{\prime} \in \bar{\mathcal{F}}$, $d_{p,\mathbf{t}_n}(f, f^{\prime}) \geqslant \epsilon$.
The $\epsilon$-packing number of $\mathcal{F}$, $\mathcal{M}\left(\epsilon, \mathcal{F}, d_{p,\mathbf{t}_n} \right)$, is the maximal cardinality of its $\epsilon$-separated subsets. 
The uniform covering and packing numbers are 
$$
\mathcal{N}_p\left(\epsilon, \mathcal{F}, n\right)=\sup_{\mathbf{t}_n \in \mathcal{T}^n}\mathcal{N} \left ( \epsilon, \mathcal{F}, d_{p,\mathbf{t}_n} \right)
$$
and
$$
\mathcal{M}_p\left(\epsilon, \mathcal{F}, n\right)=\sup_{\mathbf{t}_n \in \mathcal{T}^n}\mathcal{M} \left ( \epsilon, \mathcal{F}, d_{p,\mathbf{t}_n} \right),
$$
respectively.
\end{definition}


The capacity measures appearing last in our bounds are combinatorial dimensions.
They provide useful information about whether the class of interest uniformly satisfies the classical limit theorems \cite{Dud99}.
\begin{definition}[Fat-shattering dimension \cite{KeaSch94}, strong dimension \cite{AloBenCesHau97}]
\label{def:combinatorial-dimensions}
For $\gamma \in \mathbb{R}_+$, a subset $S = \left \{ t_i: 1 \leqslant i \leqslant n \right \}$
of $\mathcal{T}$
is said to be {\em ${\gamma}$-shattered} by $\mathcal{F}$ if
there is a function
$v:S \rightarrow \mathbb{R}$
such that, for every vector
$\mathbf{s}_n = \left ( s_i \right )_{1 \leqslant i \leqslant n}
\in \left \{ -1, 1 \right \}^n$, there is a function
$f_{\mathbf{s}_n} \in \mathcal{F}$ satisfying
$$\forall i \in \ieg 1, n \ied, \;\;
s_i \left( f_{\mathbf{s}_n} \left(t_i\right) - v\left(t_i\right) \right) \geqslant \gamma.$$ The {\em fat-shattering dimension} of
$\mathcal{F}$ at scale $\gamma$, $\gamma\mbox{-dim} \left ( \mathcal{F} \right )$,
is the maximal cardinality of a subset of $\mathcal{T}$
${\gamma}$-shattered by $\mathcal{F}$, if such a maximum exists. Otherwise, $\gamma\mbox{-dim}\left(\mathcal{F}\right)=\infty$. 
For a class $\mathcal{F}$ of integer valued functions, the notion of {\em strong dimension}, 
$S\mbox{-dim} \left(\mathcal{F} \right)$, is obtained from the definition of the fat-shattering dimension by setting $\gamma=1$ and restricting the co-domain of $v$ to $\mathbb{Z}$. 
\end{definition}

As in \cite{Gue17, Men02, MenSch04}, we make the hypothesis that the fat-shattering dimensions of the classes $\mathcal{G}_k$, $\gamma\mbox{-dim}\left(\mathcal{G}_k\right)$, grow no faster than polynomially with $\gamma^{-1}$.
\begin{hypothesis}
\label{hypothesis:restriction-gamma-dim}
Let $\mathcal{G}$ be a class of functions
satisfying Definition~\ref{def:margin-multi-category-classifiers}.
We assume that there exists a pair $\left(K_{\mathcal{G}}, d_{\mathcal{G}}\right) \in \mathbb{R}^2_+$ such that
$$
\forall \epsilon \in \left ( 0, M_{\mathcal{G}} \right ], \;\; 
\max_{1 \leqslant k \leqslant C}
\epsilon\mbox{-dim} \left(\mathcal{G}_k \right)
\leqslant K_{\mathcal{G}} \epsilon^{- d_{\mathcal{G}}}.
$$
\end{hypothesis}
Among the well-known examples of classifiers that satisfy such an assumption are support vector machines with $d_{\mathcal{G}}=2$ (Theorem~4.6 in \cite{BarSha99}) and feedforward neural networks with $d_{\mathcal{G}}=2l$ for $l$ layers (Corollary~27 in \cite{Bar98}). It should be noted that Lipschitz classifiers, such as nearest neighbours also satisfy this assumption as demonstrated by Corollary~4 in \cite{GotKonKra14}. Depending on the growth rate $d_{\mathcal{G}}$, our assumptions regarding the data are summarized in Table~\ref{tab:assumptions}.

\begin{table}[ht]
    \centering
\caption{Assumptions made on the sample size $m$ and the number of categories $C$ with respect to the growth rate $d_{\mathcal{G}}$ of the fat-shattering dimensions in Hypothesis~\ref{hypothesis:restriction-gamma-dim}.\label{tab:assumptions}}
\begin{tabular}{|l|l|}
  \hline
  Growth rate & Assumptions\\
  \hline
  $d_{\mathcal{G}} \leqslant 2$ &  $m > C > 4$ \\
  $d_{\mathcal{G}} > 2$ & $m \geqslant C^{1.2}$, $C>4$ \\ 
  \hline
\end{tabular}
\end{table}

Our starting point is the following basic supremum inequality that bounds the risk by the empirical margin risk plus a control term based on a Rademacher complexity.
\begin{theorem}[Theorem~5 in \cite{Gue17}]
Let $\mathcal{G}$ be a class of functions
satisfying Definition~\ref{def:margin-multi-category-classifiers}.
For $\gamma \in \left ( 0, 1 \right ]$,
let $\mathcal{F}_{\mathcal{G}, \gamma}$ be the class of functions deduced
from $\mathcal{G}$ according to
Definition~\ref{def:class-of-transformed-functions}. For fixed $\gamma \in (0,1]$ and $\delta \in (0,1)$, with $P^m$ probability at least $1-\delta$,
 $$
 \forall g \in \mathcal{G}, \quad L(g) \leqslant L_{\gamma,m}(g)+\frac{2}{\gamma}R_m\left(\mathcal{F}_{g, \gamma}\right)+ \sqrt{\frac{\ln(\frac{1}{\delta})}{2m}}.
 $$
\end{theorem}
We perform the following sequence of transitions between the capacity measures to derive our result.
First, we relate the empirical Rademacher complexity of $\mathcal{F}_{\mathcal{G},\gamma}$ to its metric entropy through the chaining method (see \cite{Tal14}).
More precisely, we use the following formulation of the chaining bound due to \cite{Gue17}:
\begin{align}
\hat{R}_m \left (\mathcal{F}_{\mathcal{G},\gamma} \right ) \leqslant &\  h(N) + 2\sum_{j=1}^N \left(h(j)+h(j-1)\right) \sqrt{ \frac{\ln\mathcal{N} \left (h(j), \mathcal{F}_{\mathcal{G},\gamma}, d_{2, \mathbf{z}_m} \right )}{m}}, \label{eq:chaining}
\end{align}
where $N\in\mathbb{N}^*$ and $h:\mathbb{N} \rightarrow \mathbb{R}_+$ is a decreasing function satifying $h(0) \geqslant \gamma$.
Next, using Lemma~1 in \cite{Gue17}, we decompose the metric entropy of $\mathcal{F}_{\mathcal{G},\gamma}$ in terms of the ones of the classes $\mathcal{G}_k$:
\begin{equation}\label{eq:decomposition}
\forall p \in \left[1, +\infty \right], \quad
\ln \mathcal{N} \left ( \epsilon, 
\mathcal{F}_{\mathcal{G}, \gamma}, d_{p, \mathbf{z}_m} \right ) 
\leqslant
\sum_{k=1}^C \ln \mathcal{N}
\left ( \frac{\epsilon}{C^{1/p}}, 
\mathcal{G}_k, d_{p, \mathbf{x}_m} \right ),
\end{equation}
where $\mathbf{x}_m=(x_i)_{1 \leqslant i \leqslant m} \in \mathcal{X}^m$.
Finally, our combinatorial bound derived below gives an estimate on the metric entropies of the classes $\mathcal{G}_k$ in terms of their fat-shattering dimensions.


\section{$L_p$-norm Combinatorial Metric Entropy Bound}
\label{sec:metric-entropy-bound}

We extend the $L_2$-norm metric entropy bound of \cite{MenVer03} to $L_p$-norms with $p \in \mathbb{N}^{*} \setminus \{1,2\}$. 
The bound of \cite{MenVer03} does not depend on the sample size thanks to the use of the probabilistic extraction principle.
In our extension we derive two bounds. In one of them, we keep the dependency on the sample size, 
and in the other, we remove it using the $L_p$-norm generalization of the aforementioned principle.
Under Hypothesis \ref{hypothesis:restriction-gamma-dim}, depending on the value of $d_{\mathcal{G}}$, 
the application of one or the other bound in the chaining allows us to optimize the dependency on $C$
while not degrading the ones on $m$ and $\gamma$, as will be seen in Section \ref{sec:bound-Rademacher-complexity}.

Specifically, we have the following $L_p$-norm metric entropy bounds, whose proof is given in~\ref{app:proofThm2}. 
\begin{theorem} 
\label{theo:new-Lp-norm-Sauer-Shelah-lemma}
Let $\mathcal{F}$ be a class of functions from $\mathcal{T}$ into
$\left [ -M_{\mathcal{F}}, M_{\mathcal{F}} \right ]$ with $M_{\mathcal{F}} \in [1, +\infty)$.
For $\epsilon \in \left(0, M_{\mathcal{F}} \right]$,
let $d \left(\epsilon \right) 
= \epsilon\text{-dim} \left ( \mathcal{F} \right )$.
For all values of $p \in \mathbb{N}^{*} \setminus \{1,2\}$ and $\epsilon \in \left(0, M_{\mathcal{F}} \right]$,

(a) if $n \geqslant  d\left(\frac{\epsilon}{15p}\right)$, then
$$
\ln \mathcal{N}_{p} \left ( \epsilon, \mathcal{F}, n \right )
\leqslant 2d\left(\frac{\epsilon}{15p}\right) \ln \left(\frac{15 e p n M_{\mathcal{F}}}{d\left(\frac{\epsilon}{15p}\right)\epsilon}\right);
$$ 

(b) if $n \geqslant  d\left(\frac{\epsilon}{37p}\right)$, then
$$
\ln \mathcal{N}_{p} \left ( \epsilon, \mathcal{F}, n \right )
\leqslant 10p\,d\left(\frac{\epsilon}{36p}\right) \ln\left(\frac{7 p^{\frac{1}{7}} M_{\mathcal{F}}}{\epsilon}\right).
$$ 
\end{theorem}
From \eqref{eq:decomposition} one can see that, based on
$\displaystyle{C^{\frac{1}{p}}=2^{\left(\frac{1}{p} \log_{2}(C)\right)}}$, the dependency on $C$ in the scale of covering numbers can be eliminated for all $p \geqslant \log_2(C)$. The combination of the decomposition formula \eqref{eq:decomposition} with Theorem~\ref{theo:new-Lp-norm-Sauer-Shelah-lemma} using $p=\lceil\log_2(C)\rceil$ for $C>4$ yields the following result.
\begin{corollary}\label{corollary:new-Lp-norm-Sauer-Shelah-lemma-decomposition}
Let $\mathcal{G}$ be a class of functions
as in Definition~\ref{def:margin-multi-category-classifiers}.
For $\gamma \in \left ( 0, 1 \right ]$,
let $\mathcal{F}_{\mathcal{G}, \gamma}$ be the class of functions deduced
from $\mathcal{G}$ according to
Definition~\ref{def:class-of-transformed-functions2}.
For $\epsilon \in \left ( 0, M_{\mathcal{G}} \right ]$,
let $d \left ( \epsilon \right ) = \max_{1 \leqslant k \leqslant C}
\epsilon\text{-dim} \left ( \mathcal{G}_k \right )$.
Then, for $\epsilon \in \left ( 0, \gamma \right ]$ and $C>4$,
\begin{align}
 \ln \mathcal{N}_p \left ( \epsilon, 
\mathcal{F}_{\mathcal{G}, \gamma}, m  \right )
\leqslant 2 C d\left(\frac{\epsilon}{30\log_2(2C)}\right)
\ln\left(\frac{30en\log_2\left(2C\right) M_{\mathcal{G}}}{\epsilon}\right)\label{eq:new-Lp-norm-Sauer-Shelah-lemma-decomposition-dimension-dep},
\end{align}
and 
\begin{align}
 \ln \mathcal{N}_p \left ( \epsilon, 
\mathcal{F}_{\mathcal{G}, \gamma}, m  \right )
\leqslant 10 C  \log_2\left(2C\right)d \left ( \frac{\epsilon}{72\log_2(2C)} \right)
\ln\left(\frac{14 \log^{\frac{1}{7}}_2\left(2C\right) M_{\mathcal{G}}}{\epsilon}\right). \label{eq:new-Lp-norm-Sauer-Shelah-lemma-decomposition-dimension-free}
\end{align}
\end{corollary}
\begin{proof}
Inequality \eqref{eq:new-Lp-norm-Sauer-Shelah-lemma-decomposition-dimension-dep} follows from the application of \eqref{eq:decomposition}
and part (a) of Theorem~\ref{theo:new-Lp-norm-Sauer-Shelah-lemma} (where we drop $d(\epsilon)$ from the denominator inside the logarithm as it is greater than one), 
along with the fact that $C^{1/\lceil\log_2(C)\rceil} < 2$ and $\lceil\log_2(C)\rceil < \log_2(2C)$.
We obtain Inequality \eqref{eq:new-Lp-norm-Sauer-Shelah-lemma-decomposition-dimension-free} in a similar way using part (b) of Theorem~\ref{theo:new-Lp-norm-Sauer-Shelah-lemma} instead.
\end{proof}

\section{Bound on the Rademacher complexity} \label{sec:bound-Rademacher-complexity}
As it was noted in \cite{Men02}, under Hypothesis~\ref{hypothesis:restriction-gamma-dim}, the growth rate of the fat-shattering dimension has a dramatic effect on the behavior of the Rademacher complexity of the function class. The availability of two kinds of metric entropy bounds allows us to adapt to this impact in the chaining so as to optimize the dependency on $C$ without worsening those on $m$ and $\gamma$.
Under the aforementioned hypothesis, two cases can be distinguished. For $d_{\mathcal{G}} \in (0,2)$, the formula \eqref{eq:chaining} can be upper bounded by an integral and the use of the dimension-free bound \eqref{eq:new-Lp-norm-Sauer-Shelah-lemma-decomposition-dimension-free} leads to the optimized result. For $d_{\mathcal{G}} \geqslant 2$, such a result is obtained from the application of \eqref{eq:new-Lp-norm-Sauer-Shelah-lemma-decomposition-dimension-dep} in \eqref{eq:chaining}.
The second case can also be characterized by the fact that there is a freedom in the choice of the number $N$ of steps to construct the chaining. To optimize this construction when $d_{\mathcal{G}} > 2$, we make the non-restrictive assumption that $m$ is greater than a small power of $C$.
\begin{theorem} \label{theo:rad-comp-new-bound}
Let $\mathcal{G}$ be a class of functions
as in Definition~\ref{def:margin-multi-category-classifiers}.
For $\gamma \in \left ( 0, 1 \right ]$,
let $\mathcal{F}_{\mathcal{G}, \gamma}$ be the class of functions deduced
from $\mathcal{G}$ according to
Definition~\ref{def:class-of-transformed-functions}. Then, under Hypothesis~\ref{hypothesis:restriction-gamma-dim}, there is a function $K\left(\gamma, d_{\mathcal{G}},K_{\mathcal{G}}\right)$ such that for all $C>4$,
\begin{align*}
	R_m  \left(\mathcal{F}_{\mathcal{G},\gamma} \right ) \leqslant &K\left(\gamma, d_{\mathcal{G}}, K_{\mathcal{G}}\right)\sqrt{\frac{C}{m}} \\
	& \times 
	\begin{cases} 
		\left(\ln(C)\right)^{\frac{d_{\mathcal{G}}}{2}+\frac{1}{2}} ,&\mbox{if } 0<d_{\mathcal{G}} < 2, \\
		\displaystyle{\ln(C)\ln\left(\frac{m}{C}\right)\ln^{\frac{1}{2}}\left(\frac{m\ln^{\frac{2}{3}}(C)}{C^{\frac{1}{3}}}\right)},&\mbox{if } d_{\mathcal{G}} = 2, \\
		\displaystyle{m^{\frac{1}{2}-\frac{1}{d_{\mathcal{G}}}}(\ln(C))^{2-\frac{d_{\mathcal{G}}}{2}}}\ln^{\frac{1}{2}}\left(\frac{ m^{1+\frac{1}{d_{\mathcal{G}}}}}{\ln(C)}\right),\, &\mbox{if } d_{\mathcal{G}} > 2 \; \mbox{and} \; m \geqslant C^{1.2}.
	\end{cases}
\end{align*}
\end{theorem}

Compared to Theorem~7 of \cite{Gue17}, one can see that in all three cases, the dependency on $C$ is improved:
the powers of $C$ are replaced by powers of $\ln(C)$ without losing in the dependencies on $m$ and $\gamma$.
It is interesting to note that, in the third case, when $d_{\mathcal{G}} \geqslant 4$, which is true for instance for feedforward neural networks (see Corollary~27 in \cite{Bar98}), the dependency on $C$ is slightly better than radical. This is, however, at the cost of the constant factor
$d^{d_{\mathcal{G}}}_{\mathcal{G}}$.
\begin{proof} [Proof of Theorem~\ref{theo:rad-comp-new-bound}]
 For all $j \in \mathbb{N}$, we set $h(j)=\gamma 2^{-\alpha(d_{\mathcal{G}})j}$ with $\alpha(d_{\mathcal{G}})>0$ for all $d_{\mathcal{G}} \in \mathcal{R}^{*}_{+}$ in \eqref{eq:chaining}.
 In the following, we use the relation
 \begin{align}
 \forall r>q>0, \quad \mathcal{N} \left (\epsilon, \mathcal{F}, d_{q,\mathbf{t}_n} \right) \leqslant  \mathcal{N} \left (\epsilon, \mathcal{F}, d_{r,\mathbf{t}_n} \right) \label{eq:relation-covering-numbers}
 \end{align}
which follows directly from the fact that
$$\forall f,f^{\prime} \in \mathcal{F}, \quad  d_{q, \mathbf{t}_n}(f,f^{\prime}) \leqslant d_{r, \mathbf{t}_n}(f,f^{\prime}).$$

\paragraph{First case: $d_{\mathcal{G}} \in (0,2)$}  This is the only case where Pollard's entropy condition \cite{Dud99} is satisfied. For this case we could directly use Dudley's integral formula (Formula~33 in \cite{Gue17}), however,
to optimize with respect to constants, we start from \eqref{eq:chaining} and upper bound it by an integral in the following way.

Apply \eqref{eq:relation-covering-numbers} and \eqref{eq:new-Lp-norm-Sauer-Shelah-lemma-decomposition-dimension-free} in sequence to the right-hand side of \eqref{eq:chaining} and use Hypothesis~\ref{hypothesis:restriction-gamma-dim} to get
\begin{align*}
\hat{R}_m\left(\mathcal{F}_{\mathcal{G},\gamma}\right)
\leqslant & \gamma 2^{-\alpha(d_{\mathcal{G}})N} + 2\sqrt{\frac{10C\log_2(2C)}{m}}\sum_{j=1}^N \left(\gamma 2^{-\alpha(d_{\mathcal{G}})j}+\gamma 2^{-\alpha(d_{\mathcal{G}})(j-1)}\right)
 \\ &\times \left[d\left(\frac{\gamma 2^{-\alpha(d_{\mathcal{G}})j}}{72\log_2(2C)}\right) \ln \left(\frac{14 M_{\mathcal{G}}\log^{\frac{1}{7}}_2\left(2C\right) }{\gamma 2^{-\alpha(d_{\mathcal{G}})j}}\right) \right]^{1/2} \\
\leqslant & \gamma 2^{-\alpha(d_{\mathcal{G}})N} + 2\sqrt{\frac{10C\log_2(2C) K_{\mathcal{G}}}{m}} \left(72\log_2(2C)\right)^{\frac{d_{\mathcal{G}}}{2}} \gamma^{1-\frac{d_{\mathcal{G}}}{2}}\left(1+2^{\alpha(d_{\mathcal{G}})}\right) 
\\ &\times \sum_{j=1}^N 2^{-\alpha(d_{\mathcal{G}})\left(1-\frac{d_{\mathcal{G}}}{2}\right)j} \ln^{\frac{1}{2}}\left(\frac{14 M_{\mathcal{G}} \log^{\frac{1}{7}}_2\left(2C\right)}{\gamma 2^{-\alpha(d_{\mathcal{G}})j}}\right).
\end{align*}
Letting $\displaystyle{\alpha(d_{\mathcal{G}})=\frac{2}{2-d_{\mathcal{G}}}}$, we obtain
\begin{align*}
\hat{R}_m\left(\mathcal{F}_{\mathcal{G},\gamma}\right)
\leqslant & \gamma 2^{-\frac{2}{2-d_{\mathcal{G}}}N} + 2\sqrt{\frac{10C\log_2(2C) K_{\mathcal{G}}}{m}} \left(72\log_2(2C)\right)^{\frac{d_{\mathcal{G}}}{2}}\gamma^{1-\frac{d_{\mathcal{G}}}{2}}\left(1+2^{\frac{2}{2-d_{\mathcal{G}}}}\right) \\ &\times \sum_{j=1}^N 2^{-j} \ln^{\frac{1}{2}}\left(\frac{14 M_{\mathcal{G}} \log^{\frac{1}{7}}_2\left(2C\right)}{\gamma 2^{-\frac{2}{2-d_{\mathcal{G}}}j}}\right) \\
= &\gamma 2^{-\frac{2}{2-d_{\mathcal{G}}}N} +4\sqrt{\frac{10C\log_2(2C) K_{\mathcal{G}}}{m}} \left(72\log_2(2C)\right)^{\frac{d_{\mathcal{G}}}{2}}\gamma^{1-\frac{d_{\mathcal{G}}}{2}}\left(1+2^{\frac{2}{2-d_{\mathcal{G}}}}\right)  \\ &\times \sum_{j=1}^N \left(2^{-j}-2^{-j-1}\right) \ln^{\frac{1}{2}}\left(\frac{14 M_{\mathcal{G}} \log^{\frac{1}{7}}_2\left(2C\right)}{\gamma 2^{-\frac{2}{2-d_{\mathcal{G}}}j}}\right).
\end{align*}
Taking $N \rightarrow \infty$, we can upper bound the last expression as
\begin{align*}
\hat{R}_m\left(\mathcal{F}_{\mathcal{G},\gamma}\right)
\leqslant & 4\sqrt{\frac{10C\log_2(2C) K_{\mathcal{G}}}{m}} \left(72\log_2(2C)\right)^{\frac{d_{\mathcal{G}}}{2}}\gamma^{1-\frac{d_{\mathcal{G}}}{2}}\left(1+2^{\frac{2}{2-d_{\mathcal{G}}}}\right) \\ & \times \int_{0}^{1/2} \ln^{\frac{1}{2}}\left(\frac{14 M_{\mathcal{G}} \log^{\frac{1}{7}}_2\left(2C\right)}{\gamma \epsilon^{\frac{2}{2-d_{\mathcal{G}}}}}\right) d\epsilon.
\end{align*}
Denote $K=14M_{\mathcal{G}}\log^{\frac{1}{7}}_2\left(2C\right)/\gamma$ and let us now compute the integral 
$$
\displaystyle{L=\int_{0}^{1/2} \ln^{\frac{1}{2}}\left(K/\epsilon^{\frac{2}{2-d_{\mathcal{G}}}}\right)d\epsilon=\sqrt{\frac{2}{2-d_{\mathcal{G}}}}\int_{0}^{1/2} \ln^{\frac{1}{2}}\left(\frac{K^{\frac{2-d_{\mathcal{G}}}{2}}}{\epsilon}\right)d\epsilon}.
$$
Set $\epsilon=K^{\frac{2-d_{\mathcal{G}}}{2}}e^{-t^2}$. Then,
\begin{align*}
L=\sqrt{\frac{2}{2-d_{\mathcal{G}}}}  K^{\frac{2-d_{\mathcal{G}}}{2}}  \int_{\ln^{\frac{1}{2}}\left(2K^{\frac{2-d_{\mathcal{G}}}{2}}\right)}^{\infty} 
t \cdot (2t e^{-t^2}) dt.
\end{align*}
Applying the integration by parts formula, we obtain 
\begin{align*}
L&=\sqrt{\frac{2}{2-d_{\mathcal{G}}}} K^{\frac{2-d_{\mathcal{G}}}{2}}\left(  \frac{\ln^{\frac{1}{2}}\left(2K^{\frac{2-d_{\mathcal{G}}}{2}}\right)}{2K^{\frac{2-d_{\mathcal{G}}}{2}}}
+
\int_{\ln^{\frac{1}{2}}\left(2K^{\frac{2-d_{\mathcal{G}}}{2}}\right)}^{\infty} 
e^{-t^2}dt\right) \\
&\leqslant \frac{1}{\sqrt{2(2-d_{\mathcal{G}})}}
\left(\ln^{\frac{1}{2}}\left(2K^{\frac{2-d_{\mathcal{G}}}{2}}\right)
+ \frac{1}{2\ln^{\frac{1}{2}}\left(2K^{\frac{2-d_{\mathcal{G}}}{2}}\right)} \right).
\end{align*}
Consequently,
\begin{align*}
\hat{R}_m\left(\mathcal{F}_{\mathcal{G},\gamma}\right)
\leqslant &4\sqrt{\frac{10 \cdot 72^{d_{\mathcal{G}}} \cdot K_{\mathcal{G}}}{2(2-d_{\mathcal{G}})}} \cdot \frac{\sqrt{C}(\log_2(2C))^{{1/2+d_{\mathcal{G}}/2}}}{\sqrt{m}}  \gamma^{1-\frac{d_{\mathcal{G}}}{2}}\left(1+2^{\frac{2}{2-d_{\mathcal{G}}}}\right) \\
&\times 
\left(\ln^{\frac{1}{2}}\left(2K^{\frac{2-d_{\mathcal{G}}}{2}}\right)
+ \frac{1}{2\ln^{\frac{1}{2}}\left(2K^{\frac{2-d_{\mathcal{G}}}{2}}\right)} \right).
\end{align*}

\paragraph{Second case: $d_{\mathcal{G}} \geqslant 2$}
In this case,
we apply \eqref{eq:relation-covering-numbers} and \eqref{eq:new-Lp-norm-Sauer-Shelah-lemma-decomposition-dimension-dep} to \eqref{eq:chaining} and use Hypothesis \ref{hypothesis:restriction-gamma-dim} to get
\begin{align}
\hat{R}_m\left(\mathcal{F}_{\mathcal{G},\gamma}\right)
\leqslant & \gamma 2^{-\alpha(d_{\mathcal{G}})N} + 2\sqrt{\frac{2C}{m}} \sum_{j=1}^N \left(\gamma 2^{-\alpha(d_{\mathcal{G}})j}+\gamma 2^{-\alpha(d_{\mathcal{G}})(j-1)}\right) \nonumber \\
&\times \left[d\left(\frac{\gamma 2^{-\alpha(d_{\mathcal{G}})j}}{30\log_2(2C)}\right) \ln \left(\frac{30e m  M_{\mathcal{G}}\log_2\left(2C\right) }{\gamma 2^{-\alpha(d_{\mathcal{G}})j}}\right) \right]^{1/2} \nonumber \\
\leqslant 
&\gamma 2^{-\alpha(d_{\mathcal{G}}) N} +  2\sqrt{\frac{2C K_{\mathcal{G}}}{m}} \left(30\log_2(2C)\right)^{d_{\mathcal{G}}/2} \gamma^{1-\frac{d_{\mathcal{G}}}{2}} \left(1+2^{\alpha(d_{\mathcal{G}})}\right) \nonumber\\
&\times \sum_{j=1}^N 2^{\alpha(d_{\mathcal{G}})\left(\frac{d_{\mathcal{G}}-2}{2}\right) j }
\ln^{\frac{1}{2}}\left(\frac{30 e m M_{\mathcal{G}} \log_2\left(2C\right) \cdot 2^{\alpha(d_{\mathcal{G}}) j}}{\gamma}\right). \label{eq:chaining-dg-greaterequal-2}
\end{align}
Unlike the first case, we now control the number of steps $N$ in \eqref{eq:chaining-dg-greaterequal-2} through the parameters of interest, $C$ and $m$. The aim is to optimize the dependencies with respect to them while making sure that 
(i) $N$ is a strictly positive integer and (ii) as $m \rightarrow \infty$, $N \rightarrow \infty$.

Now, if $d_{\mathcal{G}}=2$, set $\alpha(d_{\mathcal{G}})=1$. Thus, from \eqref{eq:chaining-dg-greaterequal-2}, we have
\begin{align*}
\hat{R}_m\left(\mathcal{F}_{\mathcal{G},\gamma}\right) \leqslant \gamma 2^{-N} + 180\sqrt{\frac{2C K_{\mathcal{G}}}{m}}\log_2(2C) \sum_{j=1}^N 
\ln^{\frac{1}{2}}\left(\frac{30e m  M_{\mathcal{G}}\log_2\left(2C\right) \cdot 2^{j}}{\gamma}\right). 
\end{align*}
Setting $\displaystyle{N=\left\lceil \log_2\left(\sqrt{\frac{m}{C}}\right)\right\rceil}$ and bounding the series, we obtain
\begin{align*}
\hat{R}_m \left(\mathcal{F}_{\mathcal{G},\gamma}\right) &\leqslant \gamma\sqrt{\frac{C}{m}}+180\sqrt{\frac{2C K_{\mathcal{G}}}{m}}\log_2(2C) \sum_{j=1}^N 
\ln^{\frac{1}{2}}\left(\frac{30e m M_{\mathcal{G}}\log_2\left(2C\right) \cdot 2^{j}}{\gamma}\right) \\
&< \gamma\sqrt{\frac{C}{m}} \\
&+180\sqrt{\frac{2C K_{\mathcal{G}}}{m}}\log_2(2C) \left\lceil \log_2\left(\sqrt{\frac{m}{C}}\right) \right\rceil 
\ln^{\frac{1}{2}}\left(\frac{60 e m^{3/2}\log_2\left(2C\right)M_{\mathcal{G}}}{\gamma \sqrt{C}}\right).
\end{align*}
For the final case, $d_{\mathcal{G}}>2$, we set $\displaystyle{\alpha(d_{\mathcal{G}})=\frac{2}{d_{\mathcal{G}}-2}}$ in \eqref{eq:chaining-dg-greaterequal-2} and bound the geometric series:
\begin{align}
\hat{R}_m\left(\mathcal{F}_{\mathcal{G},\gamma}\right)
\leqslant  &\gamma 2^{-\frac{2}{d_{\mathcal{G}}-2} N} +  2\sqrt{\frac{2C K_{\mathcal{G}}}{m}} \left(30\log_2(2C)\right)^{d_{\mathcal{G}}/2} \gamma^{1-\frac{d_{\mathcal{G}}}{2}} (1+2^{\frac{2}{d_{\mathcal{G}}-2}}) \nonumber \\
&\times \sum_{j=1}^N 2^j \ln^{\frac{1}{2}}\left(\frac{30 e m M_{\mathcal{G}} \log_2\left(2C\right) \cdot 2^{\frac{2}{d_{\mathcal{G}}-2} j}}{\gamma}\right) \nonumber \\
\leqslant  &\gamma 2^{-\frac{2}{d_{\mathcal{G}}-2} N} + 4 \cdot 2^N\sqrt{\frac{2C K_{\mathcal{G}}}{m}} \left(30\log_2(2C)\right)^{d_{\mathcal{G}}/2} \gamma^{1-\frac{d_{\mathcal{G}}}{2}} (1+2^{\frac{2}{d_{\mathcal{G}}-2}}) \nonumber \\
&\times \ln^{\frac{1}{2}}\left(\frac{30 e m M_{\mathcal{G}} \log_2\left(2C\right) \cdot 2^{\frac{2}{d_{\mathcal{G}}-2}N}}{\gamma}\right). \label{eq:chaining-dg-greater-2}
\end{align}
Now, let $\displaystyle{N=\left\lceil\frac{d_\mathcal{G}-2}{2d_\mathcal{G}}\log_2\left(\frac{m}{\log^{2d_{\mathcal{G}}}_2(2C)^{\frac{1}{d_{\mathcal{G}}}}}\right)\right\rceil}$. Note that, with 
the assumption $m \geqslant C^{1.2}$, $m > \log^{2d_{\mathcal{G}}}_2(2C)^{\frac{1}{d_{\mathcal{G}}}}$ for all $d_{\mathcal{G}}>2$ and thus, $N$ is a strictly positive integer. Applying it to \eqref{eq:chaining-dg-greater-2}, we get
\begin{align*}
\hat{R}_m\left(\mathcal{F}_{\mathcal{G},\gamma}\right)
\leqslant  &\frac{\gamma \log^2_2(2C)^{\frac{1}{d_{\mathcal{G}}}}}{m^{\frac{1}{d_{\mathcal{G}}}}} + 8 \sqrt{2K_{\mathcal{G}}}  \cdot 30^{d_{\mathcal{G}}/2} d^{d_{\mathcal{G}}-2}_{\mathcal{G}} \gamma^{1-\frac{d_{\mathcal{G}}}{2}} (1+2^{\frac{2}{d_{\mathcal{G}}-2}})
\\ &\times\frac{\sqrt{C}\left(\log_2(2C)\right)^{2-d_{\mathcal{G}}/2}}{m^{\frac{1}{d_{\mathcal{G}}}}}
 \ln^{\frac{1}{2}}\left(\frac{60 e d^2_{\mathcal{G}} m^{1+\frac{1}{d_{\mathcal{G}}}} M_{\mathcal{G}}}{\gamma \log_2\left(2C\right)}\right).
\end{align*}
\end{proof}

\section{Conclusions}\label{sec:conclusions}
We derived a sharper risk bound for multi-category margin classifiers following the pathway of \cite{Gue17}.
In this pathway, the first capacity measure that appears in the control term of the guaranteed risk is a Rademacher complexity. It is then related to the metric entropy through the chaining method. Using a decomposition for metric entropy, we transition from the multi-class setting to the bi-class one. Finally, a combinatorial bound gives an estimate on the metric entropy in terms of the combinatorial dimension.
The metric entropy bound used in \cite{Gue17} is the $L_2$-norm one of \cite{MenVer03}, which in this paper we generalized to $L_p$-norms with integer $p>2$.
This generalization resulted in an improved dependency on the number $C$ of categories compared to \cite{Gue17} without worsening the dependency on the sample size $m$ nor the one on the margin parameter $\gamma$.

So far, to get an explicit dependency on $C$ under minimal learnability assumptions, a transition from the multi-class case to the bi-class one has been been performed at the level of one of two capacity measures.
Realizing it at the level of a Rademacher complexity, a linear dependency on $C$ was obtained in \cite{KuzMohSye14}. In this paper, as in \cite{Gue17}, we showed that postponing it to the level of metric entropy, this dependency can be improved to a sublinear one. 
The case that remains to be studied is a decomposition at the level of a combinatorial dimension, more precisely, at that of the fat-shattering dimension. 
The goal is to complete the picture of the impact that performing a decomposition at the level of one of three different capacity measures has on the dependencies on $C$, $m$ and $\gamma$. 

\appendix
\section{Proof of Theorem~\ref{theo:new-Lp-norm-Sauer-Shelah-lemma}}
\label{app:proofThm2}
Let $\mathcal{T}_n=\{t_{i}: 1 \leqslant i \leqslant n \} \subset \mathcal{T}$ and $\boldsymbol{t}_n=(t_i)_{1 \leqslant i \leqslant n}$. Let $\mathcal{F}_{\epsilon}$ be an $\epsilon$-separated with respect to the pseudo-metric $d_{p,\boldsymbol{t}_n}$ subset of $\mathcal{F}$ of maximal cardinality. By definition, $|\mathcal{F}_\epsilon|=\mathcal{M}\left(\epsilon, \mathcal{F}, d_{p,\boldsymbol{t}_n}\right)=\mathcal{M}\left(\epsilon, \left.\mathcal{F}_{\epsilon}\right|_{\mathcal{T}_n}, d_{p,\boldsymbol{t}_n}\right)=|\left.\mathcal{F}_{\epsilon}\right|_{\mathcal{T}_n}\!|$,
where $\left.\mathcal{F}_{\epsilon}\right|_{\mathcal{T}_n}$ denotes the class $\mathcal{F}_{\epsilon}$ whose domain is restricted to $\mathcal{T}_n$. We distinguish three major steps in the proof: i) discretize functions in the set $\left.\F_{\epsilon}\right|_{\mathcal{T}_n}$, ii) demonstrate that the set of discretized functions is separated, and iii) upper bound the cardinality of the discretized set. The purpose of discretizing the set of real-valued functions is to reduce the original problem into the one that can be addressed by combinatorial means: we upper bound the packing number of the discretized set which is then related to that of the original set via the step (ii).

(a)
Let $\epsilon^{\prime}=4(4 K_p)^{1/p}$,  $\displaystyle{\eta=\frac{\epsilon}{\epsilon^{\prime}+2}}$ and $N=\left\lfloor2M_{\mathcal{F}}/\eta\right\rfloor$. Define the class $\tilde{\F}^{\eta}$ of functions from $\mathcal{T}_n$ into $\ieg 0, N \ied$ obtained by the discretization of functions in $\mathcal{F}_{\epsilon}$ in the following way:
$$
	\tilde{\F}^{\eta} = \left\{\tilde{f} : \tilde{f}(t_i) = \left\lfloor \frac{f(t_i) + M_{\mathcal{F}}}{\eta} \right\rfloor, i \in \ieg 1, n \ied, \ f \in \left.\F_{\epsilon}\right|_{\mathcal{T}_n} \right\}.
$$ 
We claim that with such a discretization, 
for any $\tilde{f}_1, \tilde{f}_2 \in \tilde{\F}^{\eta}$, $d_{p,\boldsymbol{t}_n}\left(\tilde{f}_1, \tilde{f}_2\right) \geqslant \epsilon^{\prime}$.
Using $|\lfloor a \rfloor-\lfloor b \rfloor|^p \geqslant (\max(0, |a-b|-1))^p$ for all $a,b \in \mathbb{R}_{+}$,
\begin{align*}
d_{p,\boldsymbol{t}_n}\left(\tilde{f}_1, \tilde{f}_2\right)
&=\left(\frac{1}{n}\sum_{i=1}^{n} \left|\left\lfloor\frac{f_1(t_{i}) + M_{\mathcal{F}}}{\eta}\right\rfloor-\left\lfloor\frac{f_2(t_{i}) + M_{\mathcal{F}}}{\eta}\right\rfloor \right|^p\right)^{\frac{1}{p}} \\
&\geqslant \left(\frac{1}{n} \sum_{i \in I} \left(\frac{1}{\eta}\left|f_1(t_{i})-f_2(t_{i})\right|-1\right)^p\right)^{\frac{1}{p}},
\end{align*}
where $I$ denotes the set of indices such that
$\displaystyle{\frac{1}{\eta}\left|f_1(t_{i})-f_2(t_{i})\right| \geqslant 1}$, for all $i \in I$.
Next, by the inverse triangle inequality,
$d_{p,\boldsymbol{t}_n}(f_1, f_2) \geqslant d_{p,\boldsymbol{t}_n}(f_1,0)-d_{p,\boldsymbol{t}_n}(f_2,0)$ for all $f_1, f_2 \in \mathcal{F}$, the right-hand side of the above inequality can be bounded as
\begin{align}
d_{p,\boldsymbol{t}_n}\left(\tilde{f}_1, \tilde{f}_2\right)
& \geqslant \frac{1}{\eta}\left(\frac{1}{n} \sum_{i \in I} \left|f_1(t_{i})-f_2(t_{i})\right|^p\right)^{\frac{1}{p}}-\left(\frac{|I|}{n}\right)^{\frac{1}{p}} \nonumber \\
& \geqslant \frac{1}{\eta}\left(\frac{1}{n} \sum_{i \in I} \left|f_1(t_{i})-f_2(t_{i})\right|^p\right)^{\frac{1}{p}}-1. \label{eq:distance-separation-intermediate}
\end{align}
Let $I^c$ denote the complement of $I$. Now, by definition of $\mathcal{F}_{\epsilon}$,
\begin{align*}
\frac{1}{n} \sum_{i \in I}\left|f_1(t_{i})-f_2(t_{i})\right|^p +
\frac{1}{n} \sum_{i \in I^{c}}\left|f_1(t_{i})-f_2(t_{i})\right|^p \geqslant \epsilon^p.
\end{align*}
It follows that 
\begin{align*}
\epsilon^p
&\leqslant 
\frac{1}{n} \sum_{i \in I}\left|f_1(t_{i})-f_2(t_{i})\right|^p +
\frac{|I^{c}|\eta^p}{n} \leqslant 
\frac{1}{n} \sum_{i \in I}\left|f_1(t_{i})-f_2(t_{i})\right|^p + \eta^p
\\
&\implies  
\left(\epsilon^p-\eta^p\right)^{1/p} \leqslant \left(\frac{1}{n} \sum_{i \in I}\left|f_1(t_{i})-f_2(t_{i})\right|^p\right)^{1/p}.
\end{align*}
Applying the last inequality to \eqref{eq:distance-separation-intermediate}
and using $((a-b)+b) \leqslant ((a-b)^{1/p}+b^{1/p})^p$ with $a,b \in \mathbb{R}_{+}$ and $a \geqslant b$ (where we set $a=\left(\epsilon^{\prime}+2\right)^p$ and $b=1$), we get
\begin{align*}
d_{p,\boldsymbol{t}_n}\left(\tilde{f}_1, \tilde{f}_2\right)&\geqslant \frac{1}{\eta}\left(\epsilon^p-\eta^p\right)^{1/p}-1
=\left((\epsilon^{\prime}+2)^p-1\right)^{1/p} - 1
\geqslant \epsilon^{\prime}.
\end{align*}
This proves our claim.
Then, it follows that
\begin{align}
	\mathcal{M}\left(\epsilon, \F_{\epsilon}, d_{p,\g t_n}\right) \leqslant \mathcal{M}(\epsilon^{\prime},\tilde{\F}^{\eta}, d_{p,\g t_n}) = |\tilde{\F}^{\eta}|.
	\label{eq:from-original-to-discretized}
\end{align}
The major step that remains to perform to arrive at the claimed bound is to upper bound the right-hand side of \eqref{eq:from-original-to-discretized}. To this end, we appeal to Proposition~\ref{prop:cardinality-discretized-set}.
Let $d_s$ be the strong dimension of $\tilde{\F}^{\eta}$. 
By part (1) of Lemma~3.2 in \cite{AloBenCesHau97},
\begin{align*}
	d_s &\leqslant \left(\frac{\eta}{2}\right)\text{-dim}(\left.\F_{\epsilon}\right|_{\mathcal{T}_n})
	= \left(\frac{\epsilon}{8(4 K_p)^{1/p}+4}\right)\text{-dim}(\left.\F_{\epsilon}\right|_{\mathcal{T}_n}).
\end{align*}
By Lemma~1 and the fact that $p \geqslant 3$, on the other hand, we have
$$
8(4 K_p)^{1/p}+4
< 8 \cdot 4^{1/p} p+4
< 15p.
$$
We can plug this result in the upper bound on $d_s$ based on the fact that the fat-shattering dimension decreases with the scale:
\begin{align*}
	d_s &\leqslant \left(\frac{\epsilon}{15p}\right)\text{-dim}(\left.\F_{\epsilon}\right|_{\mathcal{T}_n}) \\
	&\leqslant \left(\frac{\epsilon}{15p}\right)\text{-dim}(\F)=d\left(\frac{\epsilon}{15p}\right).
\end{align*}
Now, according to Proposition~\ref{prop:cardinality-discretized-set},
\begin{align}
	|\tilde{\F}^{\eta}|  &\leqslant  \left(\frac{e N n}{d\left(\frac{\epsilon}{15p}\right)}\right)^{2d\left(\frac{\epsilon}{15p}\right)}  \nonumber \\
	&\leqslant \left(\frac{en}{d\left(\frac{\epsilon}{15p}\right)} \left\lfloor \frac{2M_{\mathcal{F}}}{\eta}\right\rfloor\right)^{2d\left(\frac{\epsilon}{15p}\right)}  \nonumber \\
	&\leqslant \left(\frac{en}{d\left(\frac{\epsilon}{15p}\right)} \left(\frac{8M_{\mathcal{F}}(4 K_p)^{1/p}+4M_{\mathcal{F}}}{\epsilon}\right)\right)^{2d\left(\frac{\epsilon}{15p}\right)}. \label{eq:cardinality-discretized-set}
\end{align}
Applying Lemma~1 to the right-hand side of \eqref{eq:cardinality-discretized-set} and simplifying it we get
\begin{align}
	|\tilde{\F}^{\eta}| &\leqslant \left(\frac{15enM_{\mathcal{F}}p}{\epsilon d\left(\frac{\epsilon}{15p}\right)} \right)^{2d\left(\frac{\epsilon}{15p}\right)}.\label{eq:packing-bound}
\end{align}
We apply the relation \eqref{eq:from-original-to-discretized} and the following well-known inequality \cite{KolTih61}
\begin{align}
 \mathcal{N}\left(\epsilon, \F, d_{p,\g t_n}\right) \leqslant \mathcal{M}\left(\epsilon, \F, d_{p,\g t_n}\right) \label{eq:from-covering-to-packing}
\end{align} in sequence to the left-hand side of \eqref{eq:packing-bound}.
 Finally, to obtain the claimed result, we take supremum over $\g t_n \in \mathcal{T}^n$ of both sides of the obtained bound.

(b) To derive a dimension-free combinatorial bound we use the $L_p$-norm generalization of probabilistic extraction principle: Lemma~8 of \cite{Gue17}. According to this lemma, there exists a subset $\mathcal{T}_q=\{t_{i_k}: 1 \leqslant k \leqslant q \}$ of $\mathcal{T}_n$ of cardinality
\begin{align}
q \leqslant \frac{112\left(2M_{\mathcal{F}}\right)^{2p}\ln\left(|\mathcal{F}_{\epsilon}|\right)}{3 \epsilon^{2p}}, \label{eq:subvector-size}  \end{align}
such that $\mathcal{F}_{\epsilon}$ is $\epsilon_1=\epsilon/2^{\frac{p+1}{p}}$-separated with respect to $d_{p,\boldsymbol{t}_q}$, 
with $\boldsymbol{t}_q=(t_{i_k})_{1 \leqslant k \leqslant q}$. Let $\left.\mathcal{F}_{\epsilon}\right|_{\mathcal{T}_q}$ denote the class $\mathcal{F}_{\epsilon}$ whose domain is restricted to $\mathcal{T}_q$. 
We have
\begin{align}
|\mathcal{F}_\epsilon|=\mathcal{M}\left(\epsilon_1, \mathcal{F}_{\epsilon}, d_{p,\boldsymbol{t}_q}\right)=\mathcal{M}\left(\epsilon_1, \left.\F_{\epsilon}\right|_{\mathcal{T}_q}, d_{p,\boldsymbol{t}_q}\right)=|\left.\F_{\epsilon}\right|_{\mathcal{T}_q}|. \label{eq:pack-number-subvector}
\end{align}
We let $\displaystyle{\eta=\frac{\epsilon_1}{\epsilon^{\prime}+2}}$ and discretize the functions in the set $\left.\mathcal{F}_{\epsilon}\right|_{\mathcal{T}_q}$ in a similar way as in part (a):
$$
	\tilde{\F}^{\eta} = \left\{\tilde{f} : \tilde{f}(t_{i_k}) = \left\lfloor \frac{f(t_{i_k}) + M_{\mathcal{F}}}{\eta} \right\rfloor, k \in \ieg 1, q \ied, \ f \in \left.\F_{\epsilon}\right|_{\mathcal{T}_q} \right\}.
$$ 
Applying the same procedure as in the proof of part (a),
we obtain that for any $\tilde{f}_1, \tilde{f}_2 \in \tilde{\F}^{\eta}$, $d_{p,\boldsymbol{t}_q}\left(\tilde{f}_1, \tilde{f}_2\right) \geqslant \epsilon^{\prime}$, and hence
\begin{align}
	\mathcal{M}\left(\epsilon_1, \F_{\epsilon}, d_{p,\g t_q}\right) \leqslant \mathcal{M}(\epsilon^{\prime},\tilde{\F}^{\eta}, d_{p,\g t_q}) = |\tilde{\F}^{\eta}|.
	\label{eq:from-original-to-discretized-prextraction}
\end{align}
By Proposition~\ref{prop:cardinality-discretized-set}, 
$$ |\tilde{\F}^{\eta}| \leqslant \left(\frac{e N q}{d_s}\right)^{2d_s},$$
where $d_s$ is the strong dimension of $\tilde{\F}^{\eta}$.
Plugging the value of $N$ and performing similar computations as in Inequalties \eqref{eq:cardinality-discretized-set}-\eqref{eq:packing-bound} of part (a), we get
\begin{align}
	|\tilde{\F}^{\eta}| \leqslant \left(\frac{23eqM_{\mathcal{F}}(4 K_p)^{1/p}}{\epsilon d_s} \right)^{2d_s}.\label{eq:discretized-set-card-upper-bound}
\end{align}
Now, we go back from the discretized set $\tilde{\F}^{\eta}$ to $\mathcal{F}_{\epsilon}$ using the relations \eqref{eq:pack-number-subvector} and \eqref{eq:from-original-to-discretized-prextraction} which yield: $|\mathcal{F}_{\epsilon}| \leqslant |\tilde{\F}^{\eta}|$.
Using it and Inequality \eqref{eq:subvector-size} in \eqref{eq:discretized-set-card-upper-bound} give:
\begin{align*}
\ln\left(|\mathcal{F}_{\epsilon}|\right) \leqslant 2d_s\ln\left(\frac{2576 \cdot 2^{2p} e  M^{2p+1}_{\mathcal{F}}(4 K_p)^{1/p} \ln\left(| \mathcal{F}_{\epsilon}|\right)}{3\epsilon^{2p+1} d_s}\right).
\end{align*}
Now, based on $\displaystyle{\ln(u)<\sqrt{u}}$ and by a straightforward computation,
\begin{align}
\ln\left(|\mathcal{F}_{\epsilon}|\right) \leqslant 4d_s\ln\left(\frac{2576 \cdot 2^{2p+1} e  M^{2p+1}_{\mathcal{F}}(4 K_p)^{1/p} }{3\epsilon^{2p+1}}\right).\label{eq:epsilon-separated-set-card-upper-bound-simplified}
\end{align}
Next, we bound $d_s$ using part (1) of Lemma~3.2 in \cite{AloBenCesHau97} and Lemma~1:
\begin{align*}
	d_s &\leqslant \left(\frac{\eta}{2}\right)\text{-dim}(\left.\F_{\epsilon}\right|_{\mathcal{T}_q}) \\ 
	&=\left(\frac{\epsilon}{2^{\frac{4p+1}{p}} (4 K_p)^{1/p}+2^{\frac{3p+1}{p}}}\right)\text{-dim}(\left.\F_{\epsilon}\right|_{\mathcal{T}_q}) \\ 
	&\leqslant  \left(\frac{\epsilon}{16\cdot 2^{\frac{3}{p}}p+8\cdot2^{\frac{1}{p}}}\right)\text{-dim}(\left.\F_{\epsilon}\right|_{\mathcal{T}_q}) \\
	&\leqslant  \left(\frac{\epsilon}{36p}\right)\text{-dim}(\left.\F_{\epsilon}\right|_{\mathcal{T}_q}).
\end{align*}
Plugging this into \eqref{eq:epsilon-separated-set-card-upper-bound-simplified} and
applying Lemma~\ref{lemma:bound-pol-negative-order} to $K_p$, we obtain
\begin{align*}
	\ln \left(|\F_{\epsilon}|\right)
	&\leqslant 4d\left(\frac{\epsilon}{36p}\right)\ln\left(\frac{2576 \cdot 2^{2p+1} e  M^{2p+1}_{\mathcal{F}}4^{1/p} p }{3\epsilon^{2p+1}}\right) \\
	&\leqslant 10 p \, d\left(\frac{\epsilon}{36p}\right) \ln\left(\frac{7 p^{\frac{1}{7}} M_{\mathcal{F}}}{\epsilon}\right).
\end{align*}
The claim follows from the application of $|\mathcal{F}_{\epsilon}|=\mathcal{M}(\epsilon, \F, d_{p,\g t_n})$, Inequality \eqref{eq:from-covering-to-packing} and taking supremum over $\g t_n \in \mathcal{T}^n$ of both sides of the obtained bound.

\section{Technical Results}\label{app:intermediate}
\begin{lemma}\label{lemma:bound-pol-negative-order}
For all $p \in \mathbb{N}^{*} \setminus \left\{1,2\right\}$, 
$$
\sum_{k=1}^{\infty} \frac{k^p}{2^k} < p^p.
$$
\end{lemma}
\begin{proof}
By Formula (8.5) in \cite[page~119]{Ber85},
$$
\sum_{k=1}^{\infty} \frac{k^p}{u^k}=\frac{u \psi_p(-u)}{(u-1)^{(p+1)}},
$$
where $\psi_{p}(u)=\sum\limits_{j=0}^{p-1}(-1)^j \binom{p}{j+1}(u+1)^{j}\psi_{(p-1)-j}(u)$ is an Eulerian polynomial 
in $u$ of degree $p-1$ with $\psi_{0}(u)=\psi_{1}(u)=1$ (see page 116 in \cite{Ber85} for explicit form of this polynomial for smaller values of $p$).
Thus for $u=2$,
$$
\sum_{k=1}^{\infty} \frac{k^p}{2^k}=2 \psi_p(-2).
$$
We now show by induction that for all $p>2$, $\displaystyle{\psi_p(-2) < \frac{p^p}{2}}$. By definition,
\begin{align}
\psi_p(-2)=\sum\limits_{j=0}^{p-1}{p \choose j+1}\psi_{(p-1)-j}(-2). \nonumber 
\end{align}
For the base case, $p=3$, it is easily seen that $\psi_3(-2)<3^3/2$. Now, assume for $k > 3$, $\psi_{k}(-2)<k^{k}/2$. Then,
\begin{align}
\psi_{k+1}(-2)&=\sum\limits_{j=0}^{k}{k+1 \choose j+1}\psi_{k-j}(-2) \nonumber \\
&=(k+1)\psi_{k}(-2)+
\sum\limits_{j=1}^{k}{k+1 \choose j+1}\psi_{k-j}(-2) \nonumber\\
&<(k+1)k^k/2 +
\sum\limits_{j=0}^{k-1}{k+1 \choose j+2}\psi_{(k-1)-j}(-2) \nonumber\\
&=(k+1)k^k/2 +
\sum\limits_{j=0}^{k-1}\left({k \choose j+1} + {k \choose j+2} \right)\psi_{(k-1)-j}(-2). \label{eq:upper-bound-eulerian-pl-inter}
\end{align}
We have that
\begin{align*}
{k \choose j+2}&=\frac{k!}{(j+2)!(k-(j+2))!} \\
&=\frac{k!}{(j+1)!(k-(j+2))!} \cdot \frac{k-(j+1)}{(k-(j+1))(j+2)} \\
&=\frac{k!}{(j+1)!(k-(j+1))!} \cdot \frac{k-(j+1)}{j+2} \\
&< k {k \choose j+1}.
\end{align*}
Applying it in \eqref{eq:upper-bound-eulerian-pl-inter}, we obtain
\begin{align*}
\psi_{k+1}(-2) &< (k+1)k^k/2 +
\sum\limits_{j=0}^{k-1}(k+1){k \choose j+1}\psi_{(k-1)-j}(-2) \\
&<(k+1)k^k/2 + (k+1)\psi_{k}(-2)  \\
&<(k+1)k^k.
\end{align*}
Now, by the binomial theorem, for all $k>1$,
\begin{align*}
(k+1)^k &= {k \choose 0} k^0+\dots+{k \choose k-1}k^{k-1} + {k \choose k} k^k \\
&=1+\dots+k\cdot k^{k-1}+k^k \\
&>2k^k.
\end{align*}
Consequently,
\begin{align*}
\psi_{k+1}(-2) &< (k+1)\cdot (k+1)^{k}/2 =(k+1)^{k+1}/2,
\end{align*}
where we used the convention that $\displaystyle{\forall k>n, \; {n \choose k}=0}$.
\end{proof}

The results demonstrated hereafter are the generalizations of those in \cite{MenVer03}. In the following, we denote $\displaystyle{K_p=\sum\limits_{k=1}^{\infty} \frac{k^p}{2^k}}$ with $p \in \mathbb{N}^{*} \setminus \left\{1,2\right\}$. 
\begin{lemma}[After Lemma~5 of \cite{MenVer03}]\label{lemma:small-deviation}
Let $X$ be a bounded random variable. Let $M_p(X)=\left(\mathbb{E}|X|^p]\right)^{1/p}$. Then, there exist numbers $a \in \mathbb{R}$ and $\beta \in (0, 1/2]$, such that
\begin{align*}
\mathbb{P}\left\{X>a+\frac{M_p(X)}{4(2K_p)^{1/p}}\right\} \geqslant  \frac{\beta}{2} \; {\mbox and} \; \mathbb{P}\left\{X<a-\frac{M_p(X)}{4(2K_p)^{1/p}}\right\} \geqslant 1-\beta, 
\end{align*}
or vice versa.
\end{lemma}
\begin{proof}
The proof closely follows that of Lemma~5 of \cite{MenVer03} where the variance of $X$ is replaced by its higher moments.

Divide $\mathbb{R}_{+}$ into the intervals $I_k$ of length $cM_p(X)$ with $$\frac{1}{2(2K_p)^{1/p}} < c < \frac{1}{(2K_p)^{1/p}}$$ by setting 
$$I_k=(cM_p(X)k, cM_p(X)(k+1)], \quad k \geqslant 0.$$ 
Assume the lemma does not hold and let $(\beta_i)_{i \geqslant 0}$ be a non-increasing sequence of non-negative numbers such that
$$\mathbb{P}\{X > 0 \}=\beta_0\leqslant 1/2$$
and $$\mathbb{P}\{X \in I_k \}=\beta_k - \beta_{k+1}, \quad k \geqslant 0.$$
For the conclusion of the lemma to fail it should hold that
\begin{align}
\forall k \geqslant 0, \quad \beta_{k+1} \leqslant \beta_k/2. \label{eq:cond-lemma-failure}
\end{align}
Now, assume that for some $k$, $\beta_{k+1} > \beta_k/2$ and consider intervals $$J_1=(-\infty, 0] \cup (0, cM_p(X)k]=(-\infty, 0] \cup \left(\bigcup\limits_{0 \leqslant j \leqslant k-1} I_j\right)$$
and $J_2=(cM_p(X)(k+1), \infty)$. Then, $$\mathbb{P}\{X  \in J_1 \}=(1-\beta_0)+\sum\limits_{0 \leqslant j \leqslant k-1} (\beta_j-\beta_{j+1})=1-\beta_k$$ and $$\mathbb{P}\{X  \in J_2 \}=\sum\limits_{j \geqslant k+1} (\beta_j-\beta_{j+1})=\beta_{k+1}.$$
By definition of $(\beta_i)_{i \geqslant 0}$ and by our assumption, $ 1/2 \geqslant \beta_0 \geqslant \beta_k \geqslant \beta_{k+1} > \beta_k/2 \geqslant 0$, which means that $\beta_k \in (0, 1/2]$.
Now, let $a$ be the middle point between the intervals $J_1$ and $J_2$ and let $\beta=\beta_k$. We have that
\begin{align*}
cM_p(X)k=a-\frac{c M_p(X)}{2}<a-\frac{M_p(X)}{4(2K_p)^{1/p}} \implies 1-\beta \leqslant \mathbb{P}\left\{ X < a-\frac{M_p(X)}{4(2K_p)^{1/p}}\right\}
\end{align*}
and 
\begin{align*}
cM_p(X)(k+1)\!=\!a+\frac{c M_p(X)}{2}\!>\!a+\frac{M_p(X)}{4(2K_p)^{1/p}}\!\implies\!\frac{\beta}{2}\!\leqslant\!\mathbb{P}\!\left\{ X > a+\frac{M_p(X)}{4(2K_p)^{1/p}}\right\}.
\end{align*}
Thus, the lemma holds. This proves \eqref{eq:cond-lemma-failure}. Now, by induction from 
\eqref{eq:cond-lemma-failure} we get that 
$$\beta_k \leqslant 1/2^{k+1}.$$ We use it in the computation of $M^p_p(X)$. By definition, $$M^p_p(X)=\int_{0}^{\infty}\mathbb{P}\{|X| > t\}dt^p=\int_{0}^{\infty}\mathbb{P}\{X > t\}dt^p+\int_{0}^{\infty}\mathbb{P}\{X < -t\}dt^p.$$
By construction, whenever $t \in I_k$,
$\mathbb{P}\{X > t\} \leqslant \mathbb{P}\{X > cM_p(X)k \}=\mathbb{P}\{X \in \bigcup\limits_{l \geqslant k} I_l \}=\sum\limits_{l\geqslant k}\left(\beta_l-\beta_{l+1}\right)=\beta_k.$ Thus,
\begin{align*}
\int_{0}^{\infty}\mathbb{P}\{X > t\}dt^p 
&\leqslant \sum\limits_{k\geqslant 0}\int_{I_k} \beta_k p t^{p-1}dt \\
&\leqslant \left(c M_p(X)\right)^p \sum\limits_{k\geqslant 0} \frac{(k+1)^p-k^p}{2^{k+1}} \\
&\leqslant \left(c M_p(X)\right)^p \sum\limits_{k\geqslant 1} \frac{k^p}{2^k}\\
&=\left(c M_p(X)\right)^p K_p \\
&< M^p_p(X)/2.
\end{align*}
By a similar procedure, it can be proved that 
\begin{align*}
\int_{0}^{\infty}\mathbb{P}\{X < -t\}dt^p < M^p_p(X)/2.
\end{align*}
This produces a contradiction $M^p_p(X) < M^p_p(X)/2+M^p_p(X)/2=M^p_p(X)$ proving the lemma. 
\end{proof}

In the following, $\mathcal{T}=\{t_i:1 \leqslant i \leqslant n\}$ is a finite set and $\mathbf{t}_n=\left(t_i\right)_{1 \leqslant i \leqslant n}$.
\begin{lemma}[After Lemma~6 of \cite{MenVer03}]\label{lemma:separation}
Let $\mathcal{F}$ be a finite class of functions from $\mathcal{T}$ into $\left[0, M_\mathcal{F}\right]$ with $M_\mathcal{F} \in \mathbb{R}_+$ and $\left|\mathcal{F}\right|>1$. Assume that for some $\epsilon \in (0, M_\mathcal{F}]$, $\mathcal{F}$ is $\epsilon$-separated in the pseudo-metric $d_{p,\mathbf{t}_n}$.
Then there exist $i \in \ieg 1, n \ied$, $a \in \mathbb{R}$ and $\beta \in (0, 1/2]$ such that 
\begin{align*}
&\left|\left\{ f \in \mathcal{F} : f(t_i) > a + \frac{\epsilon}{8 (4K_p)^{1/p}}\right\}\right| \geqslant p_1\left|\mathcal{F}\right|\\
&\left|\left\{ f \in \mathcal{F} : f(t_i) < a - \frac{\epsilon}{8 (4K_p)^{1/p}} \right\}\right| \geqslant p_2\left|\mathcal{F}\right|,
\end{align*}
with $p_1 \geqslant \frac{\beta}{2}$ and $p_2 \geqslant 1-\beta$ or vice versa.
\end{lemma}
\begin{proof}
$\mathcal{F}$ can be viewed as a finite probability space $\left(\mathcal{F}, \mathcal{A}, P_{\mathcal{F}}\right)$ with a uniform probability measure $P_{\mathcal{F}}\left(A\right)=|A|/|\mathcal{F}|$ for any $A \in \mathcal{A}$. Then, for any two random elements $f, f^{\prime} \in \mathcal{F}$ selected independently according to $P_{\mathcal{F}}$,
\begin{align*}
	\mathbb{E}_{f,f^{\prime} \sim P_{\mathcal{F}}} \left(d_{p,\mathbf{t}_n}\left(f, f^{\prime}\right)\right)^p &= \mathbb{E}_{f,f^{\prime} \sim P_{\mathcal{F}}} \left[\frac{1}{n} \sum_{i=1}^n \left|f(t_i) - f^{\prime}(t_i)\right|^p \right]  \\
	&= \frac{1}{n} \sum_{i=1}^n \mathbb{E}_{f,f^{\prime} \sim P_{\mathcal{F}}} \left|f(t_i) - f^{\prime}(t_i)\right|^p.
\end{align*}
By the Minkowski inequality, for any $i \in \ieg 1, n \ied$,
\begin{align*}
\mathbb{E}_{f,f^{\prime} \sim P_{\mathcal{F}}} \left|f(t_i) - f^{\prime}(t_i)\right|^p &\leqslant \left(\left(\mathbb{E}_{f \sim P_{\mathcal{F}}}\left|f(t_i)\right|^p\right)^{1/p} + \left(\mathbb{E}_{f^{\prime} \sim P_{\mathcal{F}}}\left|-f^{\prime}(t_i)\right|^p\right)^{1/p}\right)^p \\
&=\left(\left(\mathbb{E}_{f \sim P_{\mathcal{F}}}\left|f(t_i)\right|^p\right)^{1/p} + \left(\mathbb{E}_{f^{\prime} \sim P_{\mathcal{F}}}\left|f^{\prime}(t_i)\right|^p\right)^{1/p}\right)^p \\
&=2^p \mathbb{E}_{f \sim P_{\mathcal{F}}}\left|f(t_i)\right|^p.
\end{align*}
Taking it into account in the formula above, we obtain,
\begin{align*}
	\mathbb{E}_{f,f^{\prime} \sim P_{\mathcal{F}}} \left(d_{p,\mathbf{t}_n}\left(f, f^{\prime}\right)\right)^p \leqslant  \frac{2^p}{n} \sum_{i=1}^n\mathbb{E}_{f \sim P_{\mathcal{F}}}\left|f(t_i)\right|^p.
\end{align*}
Now, the event that the realizations of $f$ and $f^{\prime}$ are different elements in $\mathcal{F}$ happens with probability $1 - 1/|\mathcal{F}|$. Then, by the separation assumption on $\mathcal{F}$ we have
\begin{align*}
\mathbb{E}_{f,f^{\prime} \sim P_{\mathcal{F}}} \left(d_{p,\mathbf{t}_n}\left(f, f^{\prime}\right)\right)^p
&\geqslant \left(1 - 1/|\mathcal{F}|\right) \epsilon^p 
\geqslant \left(1 - 1/2\right) \epsilon^p =\epsilon^p/2.
\end{align*}
Thus,
\begin{align*}
\frac{1}{n} \sum_{i=1}^n\mathbb{E}_{f \sim P_{\mathcal{F}}}\left|f(t_i)\right|^p \geqslant \frac{\epsilon^p}{2^{p+1}}.
\end{align*}
It means that there exists $i \in \ieg 1,n \ied$, such that 
$$
\mathbb{E}_{f \sim P_{\mathcal{F}}}\left|f(t_i)\right|^p	\geqslant \frac{\epsilon^p}{2^{p+1}}.  
$$
Next, we apply Lemma~\ref{lemma:small-deviation} to the random element $f$ and take into account that 
$$M_p(f(t_i)) \geqslant \frac{\epsilon}{2^{1+1/p}}$$ 
and that
$$
	\frac{M_p(f(t_i))}{4 (2 K_p)^{1/p}} \geqslant \frac{\epsilon}{8\times 2^{1/p}(2 K_p)^{1/p}} = \frac{\epsilon}{8(4 K_p)^{1/p}}.
$$
Then, it follows that
$$
\frac{\beta}{2} \leqslant P_{\mathcal{F}}\left\{f(t_i) > a + \frac{M_p(f(t_i))}{4 (2K_p)^{1/p}} \right\} \leqslant 
P_{\mathcal{F}}\left\{f(t_i) > a + \frac{\epsilon}{8(4K_p)^{1/p}} \right\}
$$
and, similarly,
$$
  1-\beta \leqslant P_{\mathcal{F}}\left\{f(t_i) < a - \frac{M_p(f(t_i))}{4 (2K_p)^{1/p}} \right\} \leqslant 
 P_{\mathcal{F}}\left\{f(t_i) < a - \frac{\epsilon}{8(4K_p)^{1/p}} \right\}.
$$
Finally, the claim follows from the definition of $P_{\mathcal{F}}$.
\end{proof}

The results given in the sequel call for the introduction of the definition of the $\epsilon$-separating tree.
\begin{definition}
Let $\mathcal{F}$ be a class of functions on $\mathcal{T}$. A tree $T(\mathcal{F})$ is a finite collection of subsets of $\mathcal{F}$, such
that its any two elements are either disjoint or one of them contains the other. A son of $\bar{\mathcal{F}} \in T(\mathcal{F})$ is its maximal (with respect to inclusion) proper subset. An element of $T(\mathcal{F})$ with no sons is called a leaf. Let $\epsilon>0$. If every $\bar{\mathcal{F}} \in T(\mathcal{F})$ which is not a leaf has exactly two sons $\bar{\mathcal{F}}_+, \bar{\mathcal{F}}_-$ and 
\begin{align*}
 \exists i \in \ieg 1, n \ied, \; \forall (f_+, f_-) \in  \left(\bar{\mathcal{F}}_+, \bar{\mathcal{F}}_-\right), \quad f_+(t_{i}) > f_-(t_{i})+\epsilon,
\end{align*}
then $T(\mathcal{F})$ is an $\epsilon$-separating tree.
\end{definition}

\begin{proposition}[After Proposition~8 in \cite{MenVer03}]\label{prop:proposition-8-menver03}
Let $\mathcal{F}$ be a finite class of functions from $\mathcal{T}$ into $\left[0, M_\mathcal{F}\right]$
with $M_\mathcal{F} \in \mathbb{R}_+$. Assume that for some $\epsilon \in (0, M_\mathcal{F}]$, $\mathcal{F}$ is $\epsilon$-separated in the pseudo-metric $d_{p,\mathbf{t}_n}$. Then, there is a $\epsilon/ 4(4K_p)^{1/p}$-separating tree of $\mathcal{F}$ with at least $|\mathcal{F}|^{1/2}$ leaves.
\end{proposition}
\begin{proof}
By Lemma~\ref{lemma:separation}, $\mathcal{F}$ has two subsets $\mathcal{F}_+$ and $\mathcal{F}_-$ such that 
\begin{align*}
\exists i \in \ieg 1, n \ied, \; \exists  a \in \mathbb{R}, \;\forall (f_{+},f_{-})\in \mathcal{F}_+ \times \mathcal{F}_-, \quad \begin{cases}
f_{+}(t_i) > a+\epsilon/8(4 K_p)^{1/p} \\ 
f_{-}(t_i) <a-\epsilon/ 8(4 K_p)^{1/p},
\end{cases}
\end{align*}
which implies $$f_{+}(t_i)<f_{-}(t_i)+\epsilon/4(4 K_p)^{1/p}.$$
The rest of the proof is based on induction on the cardinality of $\mathcal{F}$ and is exactly as in \cite{MenVer03}, except that the tree is now $\epsilon/ 4(4K_p)^{1/p}$-separated.
\end{proof}
\begin{proposition}[After Proposition~10 in \cite{MenVer03}]\label{prop:proposition-10-menver03}
Let $\mathcal{F}$ be a class of functions from $\mathcal{T}$ into a finite set $B$ of integers. Let $S \subseteq {\mathcal{T}}$ and let $v:S \rightarrow B$. The number of pairs $(S,v)$ strongly shattered by $\mathcal{F}$ is at least the number of leaves in any $1$-separating tree of $\mathcal{F}$.
\end{proposition}
\begin{proof}
The proof follows exactly the one of Proposition~10 in \cite{MenVer03}, with a few minor technical changes. Let $\bar{\mathcal{F}}$ be a node in a $1$-separating tree of $\mathcal{F}$. Let $N(A)$ denote the number of pairs strongly shattered by a set $A$. For the proof it suffices to show that if $\bar{\mathcal{F}}_+$ and $\bar{\mathcal{F}}_-$ are two sons of $\bar{\mathcal{F}}$, then
\begin{align}
N(\bar{\mathcal{F}}) \geqslant N(\bar{\mathcal{F}}_+)+N(\bar{\mathcal{F}}_-). \label{eq:sep-parent-child}
\end{align}
By definition of the $1$-separating tree, there exists $i_0 \in \ieg 1, n \ied$ such that
\begin{align*}
\forall (f_+, f_-) \in  \left(\bar{\mathcal{F}}_+, \bar{\mathcal{F}}_-\right), \quad f_+(t_{i_0}) > f_-(t_{i_0})+1.
\end{align*}
It follows that
\begin{align}
 \exists b \in B, \forall (f_+, f_-) \in  \left(\bar{\mathcal{F}}_+, \bar{\mathcal{F}}_-\right), \quad \begin{cases}
  f_+(t_{i_0})>b \\
  f_-(t_{i_0})<b.
 \end{cases}
  \label{eq:sep-implication} 
\end{align}
If a pair is strongly shattered either by $\bar{\mathcal{F}}_+$ or $\bar{\mathcal{F}}_-$, then it is also strongly shattered by $\bar{\mathcal{F}}$. On the other hand, if a pair $(S, v)$ is strongly shattered both by $\bar{\mathcal{F}}_+$ and $\bar{\mathcal{F}}_-$, then $t_{i_0} \not \in S$. Otherwise, there would exist $(f^{\prime}_+, f^{\prime}_-) \in  \left(\bar{\mathcal{F}}_+,\bar{\mathcal{F}}_-\right)$ satisfying $f^{\prime}_+(t_{i_0}) \leqslant v(t_{i_0})-1$ and $f^{\prime}_-(t_{i_0}) \geqslant v(t_{i_0})+1$.
Combining it with \eqref{eq:sep-implication} yields a contradiction:
$$ b+1 < v(t_{i_0}) < b-1.$$
 Now, consider a pair $\left(S \cup \{t_{i_0}\}, v^{\prime}\right)$, where $v^{\prime}(t_i)=v(t_i)$ for all $t_i \in S$ and $v^{\prime}(t_{i_0})=b$. This pair is shattered by $\bar{\mathcal{F}}$, but neither by $\bar{\mathcal{F}}_+$ or $\bar{\mathcal{F}}_-$. As $S$ is shattered both by $\bar{\mathcal{F}}_+$ and $\bar{\mathcal{F}}_-$, then from \eqref{eq:sep-implication} it follows that,
\begin{align*}
    \forall (s_i)_{1 \leqslant i \leqslant n} \in \{-1, 1\}^n, \exists f_+ \in \bar{\mathcal{F}}_+, \quad
    \begin{cases}
    \forall i \in \ieg 1,n \ied, s_i \left(f_+(t_i)-v(t_i)\right) \geqslant 1, \\
    f_+(t_{i_0})\geqslant b+1,
    \end{cases}
\end{align*}
similarly,
\begin{align*}
    \forall (s_i)_{1 \leqslant i \leqslant n} \in \{-1, 1\}^n, \exists f_- \in \bar{\mathcal{F}}_-, \quad
    \begin{cases}
    \forall i \in \ieg 1,n \ied, s_i \left(f_-(t_i)-v(t_i)\right) \geqslant 1, \\
    f_-(t_{i_0})\leqslant b-1.
    \end{cases}
\end{align*}
It proves the claim that $\bar{\mathcal{F}}$ shatters the pair  $\left(S \cup \{t_{i_0}\}, v^{\prime}\right)$.
Therefore, in both cases we get \eqref{eq:sep-parent-child}.
\end{proof}

The next result is obtained by combining Propositions~\ref{prop:proposition-8-menver03} and \ref{prop:proposition-10-menver03}.

\begin{corollary}[After Corollary 11 in \cite{MenVer03}]\label{cor:corollary11-menver03}
Let $\mathcal{F}$ be a class of functions from $\mathcal{T}$ into a finite set $B$ of integers. Let $S \subseteq {\mathcal{T}}$ and let $v:S \rightarrow B$. If $\mathcal{F}$ is $4(4 K_p)^{1/p}$-separated in the pseudo-metric $d_{p,\mathbf{t}_n}$, then it strongly shatters at least $|\mathcal{F}|^{1/2}$ pairs $(S, v)$ .
\end{corollary}

\begin{proposition}[After Proposition 12 in \cite{MenVer03}]\label{prop:cardinality-discretized-set}
Let $\mathcal{F}$ be a class of functions from $\mathcal{T}$ into $\ieg 0, b \ied$. Let $d_s=S\mbox{-dim}(\mathcal{F})$. 
Assume $\mathcal{F}$ is $4(4 K_p)^{1/p}$-separated in the pseudo-metric $d_{p,\mathbf{t}_n}$. Then for any $d \geqslant d_s$,
$$
	|\mathcal{F}| \leqslant \left(\frac{ ebn}{d}\right)^{2d}.
$$
\end{proposition}
\begin{proof}
By Corollary \ref{cor:corollary11-menver03}, $\mathcal{F}$ strongly shatters at least 
$|\mathcal{F}|^{1/2}$ pairs $(S, v)$. On the other hand, the total number of such pairs for which the cardinality of $S$ is at most $d_s$ is bounded above by
$$\sum_{k=0}^{d_s} {n \choose k} b^k.$$
To see this, note that there are at most
$\displaystyle{{n \choose k}}$ number of sets $S$ of size $k$ and for each such $S$ the number of functions $h$ is bounded above by $b^k$. 
Therefore,
$$ |\mathcal{F}|^{1/2} \leqslant \sum_{k=0}^{d_s} {n \choose k} b^k.$$
The proof is completed by bounding the right-hand side of the above inequality in a standard way as follows:
\begin{align*}
    \sum_{k=0}^{d_s} {n \choose k} b^k 
    & \leqslant \sum_{k=0}^{d} {n \choose k} b^k 
     \leqslant b^{d}\sum_{k=0}^{d} \frac{n^k}{k!}
    \leqslant b^{d}\sum_{k=0}^{d} \frac{d^k}{k!} \cdot \left(\frac{n}{d}\right)^k \\
    & \leqslant \left(\frac{bn}{d}\right)^{d}\sum_{k=0}^{d} \frac{d^k}{k!}
   \leqslant \left(\frac{enb}{d}\right)^{d},
\end{align*}
where we used the convention that for all $k>n$, $\displaystyle{{n \choose k}=0}$.
\end{proof}


\end{document}